\documentclass[12pt,reqno,b5paper]{amsproc}
\usepackage{color,anysize,float}
\usepackage{citeref}
\usepackage[noadjust]{cite}
\usepackage[colorlinks,urlcolor=blue,citecolor=blue,linkcolor=blue]{hyperref}

\usepackage{graphicx}%

\marginsize{15mm}{15mm}{1cm}{1cm}
\renewcommand{\bibitempages}[1]{}

\newtheorem{thm}{\sc Theorem}[section]

\newtheorem{lem}[thm]{\sc Lemma}
\newtheorem{prop}[thm]{\sc Proposition}
\theoremstyle{definition}

\theoremstyle{remark}
\newtheorem{rem}[thm]{\it Remark}
\newtheorem*{pf}{\it Proof}

\numberwithin{equation}{section}
\allowdisplaybreaks

\usepackage{enumerate}

\newcommand{\R}{\mathbb{R}^N}
\newcommand{\RR}{\mathbb{R}}
\newcommand{\rn}{\int_{\mathbb{R}^N}}

\begin{document}
\title[quasilinear Schr\"{o}dinger equation]{Existence and concentration of nontrivial solutions for quasilinear Schr\"{o}dinger equation with indefinite potential}

\author[Lifeng Yin, Xiaoqi Liu and Yongyong Li]{Lifeng Yin\,$^\mathrm{a}$, Xiaoqi Liu\,$^\mathrm{b}$, Yongyong Li\,$^\mathrm{c}$\vspace{-1em} }
\dedicatory{$^\mathrm{a}$School of Mathematical Sciences, Sichuan Normal University\\
 Chengdu 610039, China\\
 $^\mathrm{b}$School of Mathematics and Statistics, Anyang Normal University\\ Anyang 455002, China\\
 $^\mathrm{c}$College of Mathematics and Statistics, Northwest Normal University\\ Lanzhou, 730070, China }
\thanks{$^*$
Corresponding author. E-mail address: yin136823@163.com (L.-F. Yin)}
\maketitle
\vspace{-5mm}

\begin{abstract}
This paper is concerned with the quasilinear Schr\"{o}dinger equation
\begin{align*}
-\Delta u+V(x)u+\frac{k}{2}\Delta(u^2)u=f(u)\quad \text{in}~~\mathbb{R}^N\text{,}
\end{align*}
where $N\geq 3$, $k>0$, $V\in C(\R)$ is an indefinite potential. Under structural conditions on the potential $V$ and the nonlinearity $f$, we establish the existence of a nontrivial solution through a combination of a local linking argument, Morse theory, and the Moser iteration.  Moreover, if $f$ is odd, we obtain an unbounded sequence of nontrivial solutions via the symmetric Mountain Pass Theorem. Additionally, as $k\rightarrow0$, we analyze the concentration behavior of nontrivial solutions.
\\

\end{abstract}

\section{Introduction}
In this paper, we consider quasilinear Schr\"{o}dinger equation of the form
\begin{align}\label{eq1}
-\Delta u+V(x)u+\frac{k}{2}\Delta(u^2)u=f(u)\quad \text{in}~~\mathbb{R}^N\text{,}
\end{align}
where $k>0$, $N\geq 3$ and $V\in C(\mathbb{R}^N)$ is allowed to be indefinite.\par
It is well know that the solutions for Eq. \eqref{eq1} are closely related to the standing wave solutions $z(t,x)=e^{-i\omega t}u(x)$ to the following quasilinear Schr\"{o}dinger equation
\begin{align}\label{eq11}
  i\partial_tz=-\Delta z+W(x)z+\frac{k}{2}\Delta(|z|^2)z-\widetilde{f}(|z|^2)z\text{,} \quad (t,x)\in\mathbb{R}^+\times\mathbb{R}^N\text{,}
\end{align}
where $W: \mathbb{R}^N\rightarrow\mathbb{R}$ is a given potential and $\widetilde{f}(|u|^2)u=f(u)$ is a real function. The form of Eq. \eqref{eq11} arises in mathematical physics and is used for the superfluid film equation in plasma physics by Kurihara \cite{2}.

For the case of $k=0$, Eq. \eqref{eq1} turns out to be a semilinear Schr\"{o}dinger equation, which has been widely investigated. In recent years, mathematical studies have focused on the existence of nontrivial solutions, ground state solutions, and multiplies solutions for \eqref{eq1}, see for example \cite{MR2957647,MR3622270,MR3547688,MR2271695,Liu2024} and references therein.

For the case of $k<0$, there are few results for Eq. \eqref{eq1}. Such as, Poppenberg et al. \cite{MR1899450} considered the one-dimensional case for nonlinearity of the form $f(u)=|u|^{p-2}u$. Liu and Wang \cite{MR1933335} extended to the higher dimensional case. Note that the potential $V$ in \eqref{eq1} is given by
\begin{align}\label{v1}V(x)=W(x)-\omega\text{.} \end{align}
If the frequency $\omega$ in \eqref{v1} is large, then $V$ may be indefinite in sign. In \cite{MR3843290}, Liu and Zhou established the existence of nontrivial solution by local linking theorem and Morse theory when $V$ is the indefinite potential and $f$ satisfies some suitable conditions. For the other related results for \eqref{eq1}, we can see \cite{MR3918380,MR2659677,MR2580150,MR2927594,MR3003301,MR3327366,Huang2025,Gao2025}.\par

For $k>0$, so many papers \cite{MR3659371,MR3780725,MR3335928,MR4089641} are focused to studying
\begin{align}\label{v11}
\inf V>0\text{,}
\end{align}
i.e., the frequency $\omega$ is small. However, up to the best of our knowledge, there seem to be no results for the case $k>0$ and $V$ is an indefinite potential. We will consider that the assumption \eqref{v11} doesn't hold and $k>0$. Moreover, we will find the existence and concentration of nontrivial solutions and infinitely many solutions. \par

Next, we give the following conditions on $V$ and $f$.
\begin{enumerate}
  \item[$(V)$] $V\in C(\R)$, $\inf \limits_{x\in\R}V(x)>-\infty$, and $\lim \limits_{|x|\rightarrow\infty}V(x)=+\infty\text{.}$
  \item[$(f_1)$] $f(s)=o(s)$ as $s\rightarrow 0$.
  \item[$(f_2)$] $f\in C(\RR)$, and
   $\limsup \limits_{|s|\rightarrow\infty}\frac{f(s)}{|s|^{p-1}}=0$ for some $p\in (2,2^*)$, where $2^*=\frac{2N}{N-2}$.
  \item[$(f_3)$] There exits $\mu>\frac{8}{3}$ such that  for $s\neq 0$ there holds $0<\mu F(s)\leq sf(s)$, where $F(t)=\int_0^tf(s)ds$.
\end{enumerate}

From $(V)$, there exists a constants $m>0$ such that
\begin{align}\label{v2}\widetilde{V}(x):=V(x)+m>0\end{align}
for all $x\in\R$. For the subspace
\[E=\left\{u\in H^1(\R): \rn V(x)u^2<\infty\right\}\]
we equip the inner product
\[(u,v)=\rn \left(\nabla u\cdot\nabla v+\widetilde{V}(x)uv\right)\]
and the corresponding norm $\|u\|=(u,u)^{\frac{1}{2}}$. Obviously, the Euler-Lagrange functional associated with  Eq. \eqref{eq1} is given by
\[\mathcal{I}(u)=\frac{1}{2}\rn\left((1-ku^2) |\nabla u|^2+V(x)u^2\right)-\rn F(u)\text{.}\]
Unfortunately, the functional $\mathcal{I}$ is not well defined in $H^1(\mathbb{R}^N)$ for $N\geq 3$ and $k\neq 0$, the other difficulty is to guarantee the positiveness of the principle part, that is, $1-ku^2>0$.

We employ an approach developed in \cite{MR3780725} to introduce a variational framework associated with Eq. \eqref{eq1}. For $\alpha=\frac{9+3\sqrt{5}}{2}$, we use the following variable $g:[0,+\infty)\rightarrow\RR$,
\begin{align}\label{g1}
  g(s)=\left\{
  \begin{array}{lll}
\sqrt{1-k s^2}\text{,} \qquad \text{if}~~0\leq s<\sqrt{\frac{1}{\alpha k}}\text{,} \\
\frac{1}{\beta k s^2}+\frac{\sqrt{3}}{2}\text{,} \qquad \text{if}~~s\geq \sqrt{\frac{1}{\alpha k}}
  \end{array}
  \right.
\end{align}
where $\beta=2\alpha^{\frac{3}{2}}\sqrt{\alpha-1}$ and $g(s)=g(-s)$ if $s\leq 0$. It follows that $g\in C^1(\RR,(\frac{\sqrt{3}}{2},1])$, where $g$ is an even function, increasing in $(-\infty, 0)$ and decreasing in $[0,+\infty)$. Once $|u|_\infty\leq \sqrt{\frac1{\alpha k}}$, Eq. \eqref{eq1} can be  the following elliptic equation (see \cite{MR3010765,MR3628036,MR3688459}),
\begin{align}\label{eq3}
 - \text{div}~(g^2(u)\nabla u)+g(u)g'(u)|\nabla u|^2+V(x)u=f(u)\quad \text{in}~~\mathbb{R}^N\text{.}
\end{align}
 The energy functional associated with Eq. \eqref{eq3} is given by
\begin{align}
  \label{e1}
  \mathcal{I}_k(u)=\frac{1}{2}\rn\left( g^2(u)|\nabla u|^2+V(x)u^2\right)-\rn F(u)\text{.}
\end{align}

To prove the existence of nontrivial solutions for Eq. \eqref{eq3}, we first define a variable change that allows us to transform  Eq. \eqref{eq3} into a new problem. Let $G(t)=\int_0^tg(s) ds$, we deduce that the inverse function $G^{-1}(t)$ exists and it is an odd function, also $G, G^{-1}\in C^2$. By $(f_2)$ and the properties of transformation $u=G^{-1}(v)$, \eqref{e1} reduces to the following functional
\begin{align*}
  \mathcal{J}_k(v)=\frac{1}{2}\rn \left[|\nabla v|^2+V(x)(G^{-1}(v))^2\right]-\rn F(G^{-1}(v))\text{.}
\end{align*}
Similar to \cite[Proposition 2.4]{MR4089641}, it is easy to see that $\mathcal{J}_k$ is well defined and is of class $C^1$. Then if $v_k$ is a critical point of $\mathcal{J}_k$, then $u_k=G^{-1}(v_k)$ is a solution of Eq. \eqref{eq3}. Moreover, we study the $L^\infty$-estimate for solutions of Eq. \eqref{eq3} and find a nontrivial solution $u_k$ of Eq. \eqref{eq1} with $\|u_k\|_\infty\leq\sqrt{\frac{1}{\alpha k}}$.\par

 By $(V)$, it is a well known that the embedding $E\hookrightarrow L^q(\R)$ for $q\in [2,2^*)$ is compact, see e.g., \cite{MR1162728}. Therefore, the spectral theory of self-adjoint compact operators implies that the eigenvalue problem
\begin{align}
  \label{eigen}
  -\Delta u+V(x)u=\lambda u\text{,} \qquad u\in E\text{,}
\end{align}
possesses a complete sequence of eigenvalues
\[-\infty<\lambda_1\leq\lambda_2\leq\cdots\text{,} \qquad \lambda_i\rightarrow+\infty.\]
Each $\lambda_i$ has been repeated in the sequence according to its finite multiplicity. We denote by $\varphi_i$ the eigenfunction of $\lambda_i$ with $|\varphi_i|_2=1$.

Here are the main results of our paper for Eq. \eqref{eq1}:

\begin{thm}
  \label{th1}
  Assume that $(V)$ and $(f_1)-(f_3)$ hold. If $0$ is not an eigenvalue of $-\Delta+V$, then Eq. \eqref{eq1} has a nontrivial solution. 
\end{thm}

\begin{rem}
  To the best of our knowledge, there seems to be no literature considering the existence of nontrivial solutions for
  quasilinear Schr\"{o}dinger equation with indefinite potential. In particular, $(f_3)$ satisfies the Ambrosetti-Rabinowitz type condition. Note that in $(f_3)$ we only requrie $\mu>\frac{8}{3}$, our assumptions on the nonlinearity $f$ are weaker than those required in \cite{MR3843290}.
\end{rem}

\begin{thm}
  \label{th21}
  If $(V)$, $(f_1)-(f_3)$ hold and $0$ is not an eigenvalue of $-\Delta+V$. In addition, $f$ is odd, then
  \begin{enumerate}[$(i)$]
    \item  for any $m\in\mathbb{N}$ fixed, there exists $k_0>0$ such that Eq. \eqref{eq1} has at least $m$ weak solutions $\{u_{k,j}\}_{j=1}^m$ with $k\in(0,k_0)$.
    \item As $k\rightarrow 0^+$, $u_{k,j}$ converges to a solution $u_{0,j}$ for the following elliptic equation
  \begin{align}\label{eq33}
    -\Delta u+V(x)u=f(u)\text{,} \quad u\in H^1(\mathbb{R}^N)
 \end{align}
 with $\mathcal{I}_0(u_{0,j})\rightarrow\infty$ as $j\rightarrow\infty$.
  \end{enumerate}

\end{thm}

\begin{rem}
   Little is known about the asymptotic behavior of solutions for indefinite quasilinear Schr\"{o}dinger equations. In this work, we analyze the behavior of nontrivial solutions $u_k$ as $k\rightarrow 0^+$ and there exists a $u\in H^1(\mathbb{R}^N)$ such that $u$ is indeed a solution of Eq. \eqref{eq33}.
\end{rem}

\begin{thm}
  \label{th2}
  Suppose that $(V)$, $(f_2)$ and $(f_3)$ hold. If $f$ is odd, then Eq. \eqref{eq1} has a sequence of solutions $\{u_n\}$ such that $\mathcal{I}(u_n)\rightarrow\infty$.
\end{thm}

\begin{rem}
  Unlike Theorem \ref{th21}-$(i)$, we do not require any condition on $f$ near zero, yet we still obtain infinitely many nontrivial solutions.
\end{rem}

The paper is organized as follows. In Section 2, we provide some useful lemmas and present results in Morse theory; further, we give the proof of the Cerami condition. In section 3, using the properties of elliptic regularity and Moser iteration, we give an $L^\infty$-estimate of solutions of Eq. \eqref{eq3} and complete the proof of Theorems \ref{th1} and \ref{th21}. Finally, we will prove Theorem \ref{th2} by the symmetric mountain pass theorem.\par

\section{Preliminary Lemmas}

Throughout this paper, we will use the following notations:
\begin{itemize}
\item $H^1(\R)$ denotes the Sobolev space with norm $\|u\|_{H^1}=\left(\rn\left(|\nabla u|^2+u^2\right)\right)^{1/2}$.
\item $L^p(\mathbb{R}^N)$, $p\in[1,\infty)$, denotes a Lebesgue space, the norm $|u|_p$ by $\left(\int_{\mathbb{R}^N}|u|^p\right)^{1/p}$.
\item $S$ is the best Sobolev constant for the embedding of $D^{1,2}(\mathbb{R}^N)$ in $L^{2^{*}}(\mathbb{R}^N)$ and
\[S=\inf_{u\in D^{1,2}(\mathbb{R}^N)\backslash\{0\}}\frac{|\nabla u|_2^2}{|u|_{2^*}^2}\text{.}\]
\item $|u|_\infty:=\inf\{C>0: |u(x)|\leq C\text{~~a.e. in}~~\R\}$.
\item $\langle\cdot,\cdot\rangle$ denotes dual action between $H^1$ and $(H^1)'$.
\item $B_{r}(y):=\{x\in \mathbb{R}^{N} : |x-y|\leq r \}$ and $B_r :=\{x\in \mathbb{R}^{N} : |x|\leq r \}$.
\item $C$ denotes a positive constant that may be different in different places.
\end{itemize}

Here, we will recall some concepts on Morse theory, see, e.g., Chang \cite{MR1196690} and Mawhin-Willem \cite[Chapter 8]{MR982267}. Let $E$ be a Banach space, $\Phi: E\rightarrow\RR$ be a $C^1$-functional, $u$ is an isolated critical point and $\Phi(u)=c$. Then the $i$-th critical group of $\Phi$ at $u$ defined by
\[C_i(\Phi, u):=H_i(\Phi_c,\Phi_c\backslash\{0\})\text{,}  \quad i\in\mathbb{N}=\{0,1,2,\cdots\}\text{,} \]
where $\Phi_c:=\Phi^{-1}(-\infty, c]$ and $H_*$ stands for the singular homology with coefficients in $\mathbb{Z}$.

If $\Phi$ satisfies the Cerami condition and the critical values of $\Phi$ are bounded from below by $\alpha$, then following Bartsch-Li \cite{MR1420790}, we define the $i$-th critical group of $\Phi$ at infinity by
\[C_i(\Phi,\infty):=H_i(E, \Phi_\alpha)\text{,} \quad i\in\mathbb{N}\text{.} \]

\begin{prop}[{\cite[Proposition 3.6]{MR1420790}}]
  \label{pro2}
  If $\Phi\in C^1(E,\RR)$ satisfies the Cerami condition and $C_k(\Phi,0)\neq C_k(\Phi,\infty)$ for some $k\in\mathbb{N}$, then $\Phi$ has a nonzero critical point.
\end{prop}

\begin{prop}[{\cite[Theroem 2.1]{MR1110119}}] \label{pro3}
Suppose $\Phi\in C^1(E,\RR)$ has a local linking at $0$ with respect to the decomposition $E=Y\oplus Z$, i.e., for some $\rho>0$,
\begin{align*}
  &\Phi(u)\leq 0~~\text{\quad for}~~u\in Y\cap B_\rho\text{,} \\
  &\Phi(u)>0~~\text{\quad for}~~u\in (Z\backslash\{0\})\cap B_\rho
\end{align*}
where $B_\rho=\{u\in E :\|u\|\leq \rho\}$. If $l=\dim Y<\infty$, then $C_l(\Phi, 0)\neq 0$.
\end{prop}

The lemma below will play an important role in the proof of our main results involving the function $g$ and $G^{-1}$, which will be used later on.
\begin{lem}
  \label{lem11} The function $G^{-1}(t)$ satisfies the following properties:
  \begin{itemize}
    \item[$(i)$] $\lim \limits_{t \rightarrow 0} \frac{G^{-1}(t)}{t}=1$;
    \item[$(ii)$] $\lim \limits_{t \rightarrow\infty} \frac{G^{-1}(t)}{t}=\frac{2}{\sqrt{3}}$;
    \item[$(iii)$] $1 \leq \frac{G^{-1}(t)}{t} \leq \frac{2}{\sqrt{3}}$\quad for all $t\in\mathbb{R}$;
    \item[$(iv)$] $-1\leq \frac{t}{g(t)} g^{\prime}(t) \leq 0$\quad for all $t\in\mathbb{R}$;
    \item[$(v)$] $\frac{t}{G^{-1}(t)} \geq g\left(G^{-1}(t)\right)$\quad for all $t\in\mathbb{R}$.
  \end{itemize}
\end{lem}
\begin{proof}
  It follows from the \eqref{g1} and the Hospital's principle that
  \begin{align*}
   & \lim_{t\rightarrow0}\frac{G^{-1}(t)}{t}=\lim_{t\rightarrow0}
    \frac{1}{g(G^{-1}(t))}=1\text{,} \\
   & \lim_{t\rightarrow\infty}\frac{G^{-1}(t)}{t}=\lim_{t\rightarrow\infty}
    \frac{1}{g(G^{-1}(t))}=\frac{2}{\sqrt{3}}\text{.}
  \end{align*}
  Then $(i)$ and $(ii)$ hold. Since $g(t)$ is decreasing in $|t|$, we obtain
  \begin{align}\label{dd}G(t)\geq tg(t)\geq0\text{,} ~~t\geq0\quad\text{and}\quad G(t)<tg(t)<0\text{,} ~~t<0\text{,} \end{align}
  which implies $(v)$. By \eqref{dd}, we deduce that
  \begin{align*}
    \frac{d}{dt}\left(\frac{G^{-1}(t)}{t}\right)=\frac{t-G^{-1}(t)g(G^{-1}(t))}{g(G^{-1}(t))t^2}
    \left\{
    \begin{array}{lll}
      \geq 0\text{,} &\quad t\geq0\text{,} \\
      <0\text{,} &\quad t<0\text{.}
    \end{array}
    \right.
  \end{align*}
This, in conjunction with $(i)$ and $(ii)$, gives $(iii)$, i.e.,
\[1\leq\frac{G^{-1}(t)}{t}\leq \frac{2}{\sqrt{3}}\text{.} \]
 For $t\geq 0$, by direct calculation we obtain
\begin{align*}
 \frac{t}{g(t)}g'(t)= \left\{
  \begin{array}{lll}
    -\frac{kt^2}{1-kt^2}\text{,} &\quad 0\leq t<\sqrt{\frac{1}{\alpha k}}\text{,}\\
    \frac{-2}{1+\frac{\sqrt{3}}{2}\beta k t^2}\text{,} &\quad t\geq\sqrt{\frac{1}{\alpha k}}\text{,}
  \end{array}
  \right.
\end{align*}
which implies that
\begin{align*}
 \frac{t}{g(t)}g'(t)\geq\left\{
  \begin{array}{lll}
   -1\text{,} &\quad 0\leq t<\sqrt{\frac{1}{\alpha k}}\text{,} \\
 \frac{-2}{1+\frac{\sqrt{3}\beta}{2\alpha}}\geq -1\text{,} &\quad t\geq\sqrt{\frac{1}{\alpha k}}\text{.}
  \end{array}
  \right.
\end{align*}
For $t<0$, we show as in the case $t>0$ that $\frac{t}{g(t)}g'(t)\geq -1$. The second inequality in $(iv)$ is clear. Therefore, $(iv)$ is desired.
\end{proof}

We commence by establishing the Cerami condition for the energy functional $\mathcal{J}_k$. Firstly, we set
\[\widetilde{f}(t)=f(t)+mt\text{,}\quad \widetilde{F}(t)=\int_0^t \widetilde{f}(s)ds=F(t)+\frac{m}{2}t^2\]
and rewrite $\mathcal{J}_k$ as follow:
\begin{align*}
 \mathcal{J}_k(v)=\frac{1}{2}\rn \left(|\nabla v|^2+\widetilde{V}(x)(G^{-1}(v))^2\right)-\rn \widetilde{F}(G^{-1}(v))\text{,}
\end{align*}
where $\widetilde{V}$ is given in \eqref{v2}. Using $(f_3)$ we deduce that
\begin{align}\label{f3}
\mu \widetilde{F}(t)-\widetilde{f}(t)t\leq\left(\frac{\mu}{2}-1\right)mt^2\text{,\quad for~~} t\neq0\text{.}
\end{align}
\begin{lem}
  \label{l1}
  For $k>0$ fixed and under the assumptions $(V)$, $(f_2)$ and $(f_3)$ hold, the functional $\mathcal{J}_k$ satisfies the Cerami condition.
\end{lem}
\begin{proof}
  Let $\{v_n\}$ be a Cerami sequence of $\mathcal{J}_k$, that is, as $n\rightarrow\infty$,
  \[\mathcal{J}_k(v_n)\rightarrow c_k\text{,} \quad \left(1+\|v_n\|\right)\mathcal{J}_k'(v_n)\rightarrow 0\text{.}\]
  We claim that there exists $C>0$ such that
  \[\zeta_n:=\left[\rn \left(|\nabla v_n|^2+\widetilde{V}(x)|G^{-1}(v_n)|^2\right)\right]^{\frac{1}{2}}\leq C\text{.} \]
  Otherwise, we may assume $\zeta_n\rightarrow\infty$. Set the sequence
  \[w_n=\frac{G^{-1}(v_n)}{\zeta_n}\text{,}\]
 by $g\in C^1(\RR,(\frac{\sqrt{3}}{2},1])$ we get
  \begin{align}\label{t2}
    \|w_n\|^2&=\rn\left( |\nabla w_n|^2+\widetilde{V}(x)w_n^2\right)\notag\\
    &=\frac{1}{\zeta_n^2}\rn \left(\frac{|\nabla v_n|^2}{|g(G^{-1}(v_n))|^2}+\widetilde{V}(x)|G^{-1}(v_n)|^2\right)\notag\\
    &\leq\frac{4}{3}\text{.}
  \end{align}
  Therefore, $\{w_n\}$ is bounded in $E$. Passing to a subsequence, by the compactness of the embedding $E\hookrightarrow L^p$, we may assume that
  \begin{align}
  \label{1}w_n\rightharpoonup w~~\text{in}~~E\text{,} \quad w_n\rightarrow w~~\text{in} ~~L^p\text{,} \quad w_n(x)\rightarrow w(x) ~~\text{in} ~~\R\text{.}
  \end{align}
  Similar to Colin-Jeanjean \cite{MR2029068}, we define
  \[\psi_n=g(G^{-1}(v_n))G^{-1}(v_n)\text{.} \]
  Together with Lemma \ref{lem11}-$(iv)$ we have
  \begin{align*}
    |\nabla \psi_n|=\left(\frac{g'(G^{-1}(v_n))G^{-1}(v_n)}{g(G^{-1}(v_n))}+1\right)|\nabla v_n|\leq |\nabla v_n|\text{.}
  \end{align*}
  Moreover, by Lemma \ref{lem11}-$(iii)$ and $g\in C^1(\RR,(\frac{\sqrt{3}}{2},1])$, we have
  \[|G^{-1}(v_n)g(G^{-1}(v_n))|\leq\frac{2}{\sqrt{3}}|v_n|\text{.}\]
  So,
  \begin{align*}
    \|\psi_n\|^2&=\rn \left(|\nabla \psi_n|^2+\widetilde{V}(x)\psi_n^2\right)\\
    &\leq\rn |\nabla v_n|^2+\rn \widetilde{V}(x)g^2(G^{-1}(v_n)) |G^{-1}(v_n)|^2\\
    &\leq\frac{4}{3}\|v_n\|^2\text{.}
  \end{align*}
  Thus, using Lemma \ref{lem11}-$(iv)$ and \eqref{f3} we have
  \begin{align*}
    \mu c_k+o(1)&=\mu\mathcal{J}_k(v_n)-\langle \mathcal{J}_k'(v_n), \psi_n\rangle\\
    &=\frac{\mu}{2}\rn \left(|\nabla v_n|^2+\widetilde{V}(x)|G^{-1}(v_n)|^2\right)-\mu\rn \widetilde{F}(G^{-1}(v_n))\\
    &\quad-\!\rn \left(\nabla v_n\cdot \nabla \psi_n+ \widetilde{V}(x)\frac{G^{-1}(v_n)}{g(G^{-1}(v_n))}\psi_n\right)+\rn \frac{\widetilde{f}(G^{-1}(v_n))}{g(G^{-1}(v_n))}\psi_n\\
    &=\rn \left(\frac{\mu}{2}-\frac{g'(G^{-1}(v_n))G^{-1}(v_n)}{g(G^{-1}(v_n))}-1\right)|\nabla v_n|^2\\
    &\quad+\left(\frac{\mu}{2}-1\right)\rn \widetilde{V}(x) |G^{-1}(v_n)|^2\\
    &\quad+\rn\left(\widetilde{ f}(G^{-1}(v_n))G^{-1}(v_n)-\mu\widetilde{F}(G^{-1}(v_n))\right)\\
     &\geq\left(\frac{\mu}{2}-1\right)\zeta_n^2-\left(\frac{\mu}{2}-1\right)\rn m |G^{-1}(v_n)|^2\text{.}
  \end{align*}
Multiplying both sides by $\zeta_n^{-2}$, we obtain
  \[\left(\frac{\mu}{2}-1\right)m|w_n|_2^2=\frac{m}{\zeta_n^2}\left(\frac{\mu}{2}-1\right)\rn |G^{-1}(v_n)|^2\geq \frac{\mu}{2}-1>0\text{.}  \]
  Combining \eqref{1} and $\mu>\frac{8}{3}$, it follows that $w\neq0$. Hence, the set
  \[\Lambda=\{x\in\R\text{,~~}  w\neq 0\}\]
  is of positive Lebesgue measure. For $x\in\Lambda$, we conclude $w_n(x)\rightarrow w(x)\neq0$ and
  \[|G^{-1}(v_n)|=\zeta_n |w_n|\rightarrow\infty\text{.}\]
  Together with $(f_3)$ and the Fatou lemma we obtain
    \begin{align}\label{f1}
 \liminf_{n\rightarrow\infty}\rn \frac{\widetilde{F}(G^{-1}(v_n))}{\zeta_n^2}\geq \int_{\Lambda}\liminf_{n\rightarrow\infty}\frac{\widetilde{F}(G^{-1}(v_n))}{|G^{-1}(v_n)|^2}w_n^2=+\infty\text{.}
  \end{align}
By \eqref{f1}, for large $n$ we have
  \begin{align*}
    c_k-1\leq\mathcal{J}_k(v_n)&=\frac{1}{2}\rn\left( |\nabla v_n|^2+\widetilde{V}(x)|G^{-1}(v_n)|^2\right)-\rn \widetilde{F}(G^{-1}(v_n))\\
    &=\zeta_n^2\left(\frac{1}{2}-\rn \frac{\widetilde{F}(G^{-1}(v_n))}{\zeta_n^2}\right)\\
    &\leq\zeta_n^2\left(\frac{1}{2}-\int_{\Lambda} \frac{\widetilde{F}(G^{-1}(v_n))}{|G^{-1}(v_n)|^2}w_n^2\right)\rightarrow-\infty\text{,}
  \end{align*}
  which is a contradiction. Therefore, $\zeta_n\leq C$. Using Lemma \ref{lem11}-$(iii)$, we can obtain
  \begin{align*}
    \zeta_n^2=\rn \left(|\nabla v_n|^2+\widetilde{V}(x)|G^{-1}(v_n)|^2\right)\geq\|v_n\|^2\text{.}
  \end{align*}
  Then, $\{v_n\}$ is bounded in $E$. Similar to \eqref{1}, along a subsequence we obtain
  \[v_n\rightharpoonup v~~\text{in}~~E\text{,} \quad v_n\rightarrow v~~\text{in}~~L^p~\text{for all ~~}2\leq p<2^*\text{,\quad} v_n\rightarrow v\text{~a.e.~ in}~\mathbb{R}^N\text{.}\]
We claim that there exists a constant $\eta>0$ such that
\begin{align}\label{2}h_n=\left(\frac{G^{-1}(v_n)}{g(G^{-1}(v_n))}-\frac{G^{-1}(v)}{g(G^{-1}(v))}\right)(v_n-v)\geq\eta(v_n-v)^2\text{.} \end{align}
In fact, by $(iv)$ in Lemma \ref{lem11} and $g\in C^1(\RR,(\frac{\sqrt{3}}{2},1])$, we deduce
\[\frac{d}{ds}\left(\frac{G^{-1}(s)}{g(G^{-1}(s))}\right)=
\frac{1-\frac{G^{-1}(s)g'(G^{-1}(s))}{g(G^{-1}(v_n))}}{g^2(G^{-1}(s))}>0\text{.} \]
Hence, $\frac{G^{-1}(s)}{g(G^{-1}(s))}$ is strictly increasing, and for each $C>0$ there exists $\eta>0$ such that
\[\frac{d}{ds}\left(\frac{G^{-1}(s)}{g(G^{-1}(s))}\right)\geq\eta\text{,\quad}  \text{when}~~|s|\leq C\text{.} \]
So, by the Mean Value Theorem, we obtain
\begin{align*}
h_n\geq \eta(v_n-v)^2\text{.}
\end{align*}
So, the claim \eqref{2} holds. By $(f_1)$ and $(f_2)$, we get
\begin{align}
\label{grow}|f(t)|\leq \varepsilon |t|+C_\varepsilon |t|^{p-1}\text{,\quad }p\in (2,2^*)\text{.}
\end{align}
From the properties of $G^{-1}$, $g\in C^1\left(\mathbb{R},\left(\frac{\sqrt{3}}{2},1\right]\right)$,  \eqref{grow} and the H\"{o}lder inequality, it follows that
  \begin{align}\label{3}
 &\quad\rn \left(\frac{\widetilde{f}(G^{-1}(v_n))}{g(G^{-1}(v_n))}-\frac{\widetilde{f}(G^{-1}(v))}{g(G^{-1}(v))}\right) (v_n-v)\notag\\
 &\leq C\rn (|v_n|^{p-1}+|v_n|+|v|^{p-1}+|v|)|v_n-v|\notag\\
    &\leq C(|v_n|_2+|v|_2)|v_n-v|_2+C(|v_n|_{p}^{p-1}+|v|_p^{p-1})|v_n-v|_p\notag\\
    &=o(1)\text{.}
  \end{align}
From \eqref{2} and \eqref{3}, it follows that
  \begin{align*}
    o(1)&=\langle \mathcal{J}_k'(v_n)-\mathcal{J}_k'(v), v_n-v\rangle\\
    &=\rn |\nabla (v_n-v)|^2+\widetilde{V}(x)\left(\frac{G^{-1}(v_n)}{g(G^{-1}(v_n))}
    -\frac{G^{-1}(v)}{g(G^{-1}(v))}\right)(v_n-v)\\
    &\quad-\rn \left(\frac{\widetilde{f}(G^{-1}(v_n))}{g(G^{-1}(v_n))}-\frac{\widetilde{f}
    (G^{-1}(v))}{g(G^{-1}(v))}\right)(v_n-v)\\
    &\geq C\|v_n-v\|^2+o(1)\text{.}
  \end{align*}
  This implies that $v_n\rightarrow v$ in $E$.
\end{proof}

\section{Proof of Theorem \ref{th1}}

In this section, we will investigate the critical groups of $\mathcal{J}_k$ at the origin and infinity. Using Proposition \ref{pro2}, we conclude the nonzero critical points $v_k$ of $\mathcal{J}_k$. Moreover, we prove the $L^\infty$-estimate to find a nontrivial solution of Eq. \eqref{eq1}.\par
Set $\lambda_0=-\infty$, since $0$ is not an eigenvalue of \eqref{eigen}, we may assume that there exists $\nu\geq 0$ such that $0\in(\lambda_\nu, \lambda_{\nu+1})$. For $\nu\geq1$ we set
\[E^-=\text{span}\{\varphi_1\text{,} \cdots\text{,} \varphi_\nu\}\text{,\qquad} E^+=(E^-)^\perp\text{.}\]
Then $E^-$ and $E^+$ are the negative space and positive space of the quadratic form
\[\mathcal{B}(u)=\frac{1}{2}\rn\left( |\nabla u|^2+V(x)u^2\right)\]
respectively with $\dim E^-=\nu$. Moreover, there is a constant $\kappa>0$ such that
\begin{align}\label{s1}\pm \mathcal{B}(u)\geq \kappa\|u\|^2\text{,\quad} u\in E^{\pm}.\end{align}

\begin{lem}
  \label{l2}
  The functional $\mathcal{J}_k$ has a local linking at $0$ to the decomposition $E=E^-\oplus E^+$.
\end{lem}
\begin{proof}
  Define the functional
  \[Q(v)=\frac{1}{2}\rn \left(|\nabla v|^2+V(x)|G^{-1}(v)|^2\right)\text{.}\]
 The principle part $Q$ of $\mathcal{J}_k$ is a $C^2$-functional on $E$, with derivatives given by
 \begin{align*}
   &\langle Q'(v), \phi\rangle=\rn \left(\nabla v\cdot\nabla \phi+V(x)\frac{G^{-1}(v)}{g(G^{-1}(v))}\phi\right)\text{,}\\
   &\langle Q''(v)\phi, \psi\rangle=\rn \left(\nabla \phi\cdot\nabla\psi+V(x)\frac{1-\frac{g'(G^{-1}(v))G^{-1}(v)}{g(G^{-1}(v))}}{(g(G^{-1}(v)))^2}\phi\psi\right)
 \end{align*}
 for all $v,\phi,\psi\in E$. In particular, since $G^{-1}(0)=0$ and $g(G^{-1}(0))=1$, we have $Q(0)=0$ and
 \[\langle Q''(0)\phi, \psi\rangle=\rn\left(\nabla \phi\cdot\nabla\psi+V(x)\phi\psi\right)\text{.}\]
 Applying the Taylor formula, we have
 \begin{align}\label{6}
 Q(v)&=Q(0)+\langle Q'(0), v\rangle+\frac{1}{2}\langle Q''(0)v, v\rangle+o(\|v\|^2)\notag\\
 &=\frac{1}{2}\rn\left( |\nabla v|^2+V(x)v^2\right)+o(\|v\|^2)
 \end{align}
 as $\|v\|\rightarrow0$. From $(f_1), (f_2)$ and $(iii)$ in Lemma \ref{lem11}, for any $\varepsilon>0$, there exists $C_\varepsilon>0$ such that
 \begin{align}
   \label{5x}
   |F(G^{-1}(v))|\leq \varepsilon v^2+C_\varepsilon |v|^p\text{,\quad for}~~p\in(2,2^*)\text{.}
 \end{align}
 Thus, from \eqref{5x} we have
\begin{align}
\label{7}\rn F(G^{-1}(v))=o(\|v\|^2)~\text{\quad as}~\|v\|\rightarrow0\text{.}
\end{align}
 By \eqref{6} and \eqref{7},  as $\|v\|\rightarrow 0$ we conclude
 \begin{align*}
   \mathcal{J}_k(v)&=Q(v)-\rn F(G^{-1}(v))\\
   &=\frac{1}{2}\rn\left(|\nabla v|^2+V(x)v^2\right)+o(\|v\|^2)\\
   &=\mathcal{B}(v)+o(\|v\|^2)\text{.}
 \end{align*}
 From this and \eqref{s1}, it is easy to see that $\mathcal{J}_k$ has a local linking at $0$.
\end{proof}

In consequence of Proposition \ref{pro2}, we shall investigate the critical groups of $\mathcal{J}_k$ at infinity. Firstly, we need the following lemma, which plays an important role in proving $C_k(\mathcal{J}_k,\infty)$.
\begin{lem}
  \label{l4}
  There exists $A>0$ such that $\frac{d}{dt}\big|_{t=1}\mathcal{J}_k(tv)<0$ if $\mathcal{J}_k(v)\leq -A$.
\end{lem}
\begin{proof}
  Otherwise, there is a sequence $\{v_n\}\subset E$ such that $\mathcal{J}_k(v_n)\leq -n$, but
  \begin{align}\label{271}
   \frac{d}{dt}\Bigg|_{t=1}\mathcal{J}_k(t v_n)&=\langle \mathcal{J}_k'(v_n), v_n\rangle\geq 0\text{.}
  \end{align}
  Note that, $f(s)s\geq 0$, by Lemma \ref{lem11} $(iii)$, $g\in C\left(\mathbb{R},\left(\frac{\sqrt{3}}{2}, 1\right]\right)$ and \eqref{271}, we obtain
  \begin{align}
    \label{27}
\langle \mathcal{J}_k'(v_n), v_n\rangle&=\rn \left(|\nabla v_n|^2+\widetilde{V}(x)\frac{G^{-1}(v_n)v_n}{g(G^{-1}(v_n))}\right)-\rn \widetilde{f}(G^{-1}(v_n))\frac{v_n}{g(G^{-1}(v_n))}\notag\\
  &=\rn \left(|\nabla v_n|^2+\widetilde{V}(x)G^{-1}(v_n)^2\frac{v_n}{G^{-1}(v_n)g(G^{-1}(v_n))}\right)\notag\\
  &\quad-\rn \widetilde{f}(G^{-1}(v_n))G^{-1}(v_n)\frac{v_n}{G^{-1}(v_n)g(G^{-1}(v_n))}\notag\\
  &\leq\frac{2}{\sqrt{3}}\rn \left(|\nabla v_n|^2+\widetilde{V}(x)G^{-1}(v_n)^2\right)-\frac{\sqrt{3}}{2}\rn\widetilde{f}(G^{-1}(v_n))G^{-1}(v_n)\text{.}
  \end{align}
By \eqref{27} and \eqref{f3}, we obtain
  \begin{align}
    \label{29}
    0&\geq \frac{3\mu}{4}\mathcal{J}_k(v_n)-\frac{\sqrt{3}}{2}\langle \mathcal{J}_k'(v_n), v_n\rangle\notag\\
    &\geq\left(\frac{3\mu}{8}-1\right)\rn\left(|\nabla v_n|^2+\widetilde{V}(x)|G^{-1}(v_n)|^2\right)\notag\\
    &\quad+\frac{3}{4}\rn \left(\widetilde{f}(G^{-1}(v_n))G^{-1}(v_n)-\mu\widetilde{ F}(G^{-1}(v_n))\right)\notag\\
    &\geq\left(\frac{3\mu }{8}-1\right)\rn\left(|\nabla v_n|^2+\widetilde{V}(x)|G^{-1}(v_n)|^2\right)\notag\\
    &\quad-\left(\frac{3\mu}{8}-\frac{3}{4}\right)\rn m G^{-1}(v_n)^2\text{.}
  \end{align}
  We claim that
  \[\zeta_n=\left\{\rn\left(|\nabla v_n|^2+\widetilde{V}(x)|G^{-1}(v_n)|^2\right)\right\}^{\frac{1}{2}}\rightarrow\infty\text{.}\]
 Otherwise, up to a subsequence, there exists a constant $C>0$ such that $\zeta_n\leq C$. By Lemma \ref{lem11}-$(iii)$, we obtain that $\{v_n\}$ is bounded in $E$.  Thus, $\mathcal{J}_k(v_n)$ is bounded, contradicting the fact that $\mathcal{J}_k(v_n)\leq -n$.
Set
  \[w_n=\frac{G^{-1}(v_n)}{\zeta_n}\text{.}\]
  Then $\{w_n\}$ is bounded $E$. Along with a subsequence, we may assume that
  \begin{align}\label{t4}w_n\rightharpoonup w~~ \text{in}~~E\text{,\quad }w_n\rightarrow w ~~\text{in}~~L^2\text{,\quad} w_n\rightarrow w ~~\text{a.e. in }~~\R\text{.} \end{align}
 Multiplying both sides of \eqref{29} by $\zeta_n^{-2}$, we have
 \[\left(\frac{3\mu}{8}-\frac{3}{4}\right)m\rn |w_n|^2\geq\left(\frac{3\mu }{8}-1\right)\text{.}\]
 This implies that $w\neq \mathbf{0}$ since $\mu>\frac{8}{3}$ and \eqref{t4}. Similar to \eqref{f1}, we obtain
 \begin{align}
   \label{f2}
   \frac{1}{\zeta_n^2}\rn \widetilde{f}(G^{-1}(v_n))G^{-1}(v_n)&=\rn \frac{\widetilde{f}(G^{-1}(v_n))G^{-1}(v_n)}{G^{-1}(v_n)^2}w_n^2\notag\\
   &\geq\rn \frac{f(G^{-1}(v_n))G^{-1}(v_n)}{G^{-1}(v_n)^2}w_n^2+m\rn w_n^2\notag\\
   &\geq\rn \frac{\mu F(G^{-1}(v_n))}{G^{-1}(v_n)^2}w_n^2+m\rn w_n^2\rightarrow\infty\text{.}
 \end{align}
 From \eqref{27}, Fatou Lemma and \eqref{f2}, it follows that
 \begin{align*}
   0&\leq\langle \mathcal{J}_k'(v_n), v_n\rangle\\
   &=\rn\left(| \nabla v_n|^2+\widetilde{V}(x)\frac{G^{-1}(v_n)v_n}{g(G^{-1}(v_n))}\right)-\rn \widetilde{f}(G^{-1}(v_n))\frac{v_n}{g(G^{-1}(v_n))}\\
   &\leq\frac{2}{\sqrt{3}}\rn\left(| \nabla v_n|^2+\widetilde{V}(x)(G^{-1}(v_n))^2\right)-\frac{\sqrt{3}}{2}\rn\widetilde{f}(G^{-1}(v_n))G^{-1}(v_n)\\
   &\leq\zeta_n^2\left\{\frac{2}{\sqrt{3}}-\frac{\sqrt{3}}{2}\rn\frac{\widetilde{f}(G^{-1}(v_n))G^{-1}(v_n)}{\zeta_n^2}\right\}\rightarrow-\infty\text{,}
 \end{align*}
 which is a contradiction. Thus, the conclusion of the Lemma \ref{l4} must be true.
\end{proof}

\begin{lem}
  \label{l5}
  Under our assumptions, $C_l(\mathcal{J}_k,\infty)\cong 0$ for $l\in \mathbb{N}$.
\end{lem}

\begin{proof}
  Let $B$ be the unit ball in $E$ and $S=\partial B$. Without loss of generality, we may assume that
  \begin{align*}
    -A<\inf_{\|v\|\leq 2}\mathcal{J}_k(v)\text{,}
  \end{align*}
where $A$ is given in Lemma \ref{l4}. Since $|G^{-1}(v)|\leq \frac{2}{\sqrt{3}}|v|$, for $w\in S$, reasoning as \eqref{f1} we deduce
\begin{align*}
  \mathcal{J}_k(sw)=&\frac{1}{2}\rn \left(|\nabla(sw)|^2+\widetilde{V}(x)|G^{-1}(sw)|^2\right)-\rn \widetilde{F}(G^{-1}(sw))\\
  \leq&\frac{2s^2}{3}\rn \left(|\nabla w|^2+\widetilde{V}(x)w^2\right)-\rn \widetilde{F}(G^{-1}(sw))\\
  \leq&s^2\left(\frac{2}{3}-\rn \frac{\widetilde{F}(G^{-1}(sw))}{s^2}\right)\rightarrow-\infty,\quad \text{as}~~s\rightarrow\infty\text{.}
\end{align*}
Thus, there exists $s_w>0$ such that
\[\mathcal{J}_k(s_ww)=-A\text{.}\]
Define $v=s_ww$, we obtain
\begin{align*}
  \frac{d}{ds}\Bigg|_{s=s_w}\mathcal{J}_k(sw)=\frac{1}{s_w}\frac{d}{dt}\Bigg|_{t=1}\mathcal{J}_k(tv)<0\text{.}
\end{align*}
Therefore, by the implicit function theorem, the map
\[w\mapsto s_w\]
It is a continuous function on $S$. Using this map we can construct a deformation from $X\backslash B$ to $(\mathcal{J}_k)_{-A}$. Then, we deduce via the homotopic invariance of singular homology
\[C_l(\mathcal{J}_k,\infty)=H_l(E, (\mathcal{J}_k)_{-A})\cong H_l(E, E\backslash B)=0\text{,\quad}\text{for}~~ l\in\mathbb{N}\text{.}\]
The proof of the lemma is complete.
\end{proof}

\subsection{$L^\infty$-Estimate of the solutions}

Here, using standard elliptic regularity\cite{MR1814364}, we will establish an estimate from above for $\|v_k\|_\infty$, where $v_k$ is a solution of Eq. \eqref{eq3}. However, this information is not enough to obtain that $\|v_k\|\leq\sqrt{\frac{1}{\alpha k}}$, where $\alpha$ is given in section 2. To this end, we must verify that the Sobolev norm of $v_k$ is uniformly bounded, independent of $k>0$.

Using $u=G^{-1}(v)$, we define the set of critical points as follows:
\[\mathcal{N}_k=\{v\in E\setminus \{0\}: \mathcal{J}_k'(v)=0\}\text{,\quad}\mathcal{M}_k=\{u\in E\setminus \{0\}: \mathcal{I}_k'(u)=0\}.\] %
Denote
\[b_k=\inf_{v\in \mathcal{N}_k}\mathcal{J}_k(v)\text{,\quad} m_k=\inf_{u\in \mathcal{M}_k}\mathcal{I}_k(u).\] %
It is obvious that $b_k=m_k$ for $k>0$ fixed.
\begin{lem}\label{bk}
  There exists a positive constant $b>0$ such that $b_k\leq b$ for all $k>0$.
\end{lem}

\begin{pf}
  Let $\varphi\in E^+$ fixed. Denote $\Psi(t)=\mathcal{I}_k(t\varphi)$ and $C_\varphi=\rn |\nabla \varphi|^2+V(x)\varphi^2$. It is clear that $C_\varphi>0$. Let $t\geq 1$. According to \eqref{F1}, we have
  \begin{align*}
    \Psi(t)&\leq\frac{t^2}{2}C_{\varphi}-\int_{|\varphi|\geq 1}|t\varphi|^{\mu}\\
    &\leq t^2\left(\frac{1}{2}C_{\varphi}-t^{\mu-2}\int_{|\varphi|\geq 1}|\varphi|^\mu\right)
  \end{align*}
  which implies $\Psi(t)<0$ if $t>T_*:=\left(\frac{C_\varphi}{2\int_{|\varphi|\geq1}|\varphi|^\mu}\right)^{\frac{1}{\mu-2}}$. Then
  \begin{align*}
    b_k=m_k\leq \max_{t\in[0,T_*]}\Psi(t)\leq \max_{t\in[0,T_*]}\left(\frac{t^2}{2}C_{\varphi}-t^{\mu}\int_{|\varphi|\geq 1}|\varphi|^\mu\right)\leq b\text{.}
  \end{align*}
\end{pf}
\begin{lem}
  \label{lemm}
  The solution $v_k$ of Eq. \eqref{eq3} satisfies $\|v_k\|\leq C$.
\end{lem}

\begin{proof}
Let $v_k$ be solution of Eq. \eqref{eq3} such that
\[\mathcal{J}_k(v_k)=b_k<+\infty\text{,\quad} \mathcal{J}_k'(v_k)=0\text{.}\]
We claim:
\[\zeta_k^2=\rn \left(|\nabla v_k|^2+\widetilde{V}(x)|G^{-1}(v_k)|^2\right)\leq C\text{.}\]
Otherwise, there exist a subsequence $\{v_{k_n}\}$ of $\{v_k\}$ such that
\[\zeta_{k_n}^2\rightarrow+\infty\text{,\qquad} \mathcal{J}_{k_n}'(v_{k_n})=0\text{,\qquad} \mathcal{J}_{k_n}(v_{k_n}):=b_{k_n}\text{.}\]
Taking $\psi_n=g(G^{-1}(v_{k_n}))G^{-1}(v_{k_n})$, by Lemma \ref{bk} we have
\begin{align*}
    \mu b+o(1)&\geq\mu b_{n_k}=\mu\mathcal{J}_{k_n}(v_{k_n})-\langle \mathcal{J}_{k_n}'(v_{k_n}), \psi_n\rangle\\
    &=\rn \left(\frac{\mu}{2}-\left(\frac{g'(G^{-1}(v_{k_n}))G^{-1}(v_{k_n})}{g(G^{-1}(v_{k_n}))}+1\right)\right)|\nabla v_{k_n}|^2\\
    &\quad+\left(\frac{\mu}{2}-1\right)\rn \widetilde{V}(x) |G^{-1}(v_{k_n})|^2\\
    &\quad+\rn\left(\widetilde{ f}(G^{-1}(v_{k_n}))G^{-1}(v_{k_n})-\mu\widetilde{F}(G^{-1}(v_{k_n}))\right)\\
     &\geq\left(\frac{\mu}{2}-1\right)\zeta_{k_n}^2-\left(\frac{\mu}{2}-1\right)\rn m |G^{-1}(v_{k_n})|^2\text{.}
  \end{align*}
Consider the sequence
  \[w_n=\frac{G^{-1}(v_{k_n})}{\zeta_{k_n}}.\]
The remaining proof is similar to the proof of Lemma \ref{l1}. Thus, we show that $\zeta_k^2\leq C$. Using $(iii)$ of Lemma \ref{lem11} again, $\|v_k\|\leq C$ is independent of $k$.
\end{proof}
\begin{lem}
  \label{lem41}
 If $v_k$ is a solution of \eqref{eq3}, then there exists $C_0>0$ independent of $k>0$ such that $\|v_k\|_\infty\leq C_0$.
\end{lem}
\begin{proof}
  For any $r>0$, let $B_1=B_r(0)$, $B_2=B_{2r}(0)$. Noting that $v_k$ is a weak solution of the following problem
  \begin{align*}
    -\Delta u=\left(\frac{f(G^{-1}(u))}{g(G^{-1}(u))u}-V(x) \frac{G^{-1}(u)}{g(G^{-1}(u))u}\right)u \qquad\text{in}~~~~\R\text{.}
  \end{align*}
Let
\[a(x):=\frac{f(G^{-1}(v_k))}{g(G^{-1}(v_k))v_k}-V(x) \frac{G^{-1}(v_k)}{g(G^{-1}(v_k))v_k}\text{.}\]
By \eqref{grow}, $g\in C^1(\mathbb{R}, (\frac{\sqrt{3}}{2},1])$, $V\in C(\mathbb{R}^N)$, $(iii)$ in Lemma \ref{lem11} and Lemma \ref{lemm}, we have
\begin{align*}
\int_{B_2}|a(x)|^{\frac{N}{2}}&\leq C\int_{B_2}\left(\left|\frac{f(G^{-1}(v_k))}{v_kg(G^{-1}(v_k))}\right|^{\frac{N}{2}}+\left|V(x) \frac{G^{-1}(v_k)}{g(G^{-1}(v_k))v_k}\right|^{\frac{N}{2}}\right)\notag\\
&\leq C\int_{B_2}|v_k|^{(p-2)\frac{N}{2}}+C|B_2|\notag\\
&\leq C|B_2|+C\left(\int_{B_2}|v_k|^p\right)^{\frac{N(p-2)}{2p}}|B_2|^{\frac{2p-N(p-2)}{2p}}<+\infty\text{.}
\end{align*}
The Br\'{e}zis-Kato Theorem implies that $v_k\in L^q(B_2)$ for any $1\leq q<+\infty$. Note that
\begin{align*}
\int_{B_2}|a(x)v_k|^q&=\int_{B_2}\left|\frac{f(G^{-1}(v_k))}{g(G^{-1}(v_k))}-V(x) \frac{G^{-1}(v_k)}{g(G^{-1}(v_k))}\right|^q\\
&\leq C\int_{B_2}|v_k|^q+C\int_{B_2} |v_k|^{q(p-1)}<+\infty\text{,}
\end{align*}
we have
\[-\Delta v_k=a(x)v_k\in L^q(B_2)\text{.}\]
By the Calderon-Zygmund inequality and Schauder estimates, see e.g. \cite[Theorem 9.9]{MR1814364}, one has $v_k\in W^{2,q}(B_2)$. It follows from classical interior $L^q$-estimates that
\begin{align}
  \label{t5}
  \|v_k\|_{W^{2,q}(B_1)}\leq C(|a(x)v_k|_{L^q(B_2)}+|v_k|_{L^q(B_2)})\text{,}
\end{align}
where $C=C(r,q)$. For $q\geq 2$, by the Sobolev embedding theorem we get $v_k\in C^{1,\alpha}(\overline{B}_1)$ for some $\alpha\in(0,1)$. Subsequently, we deduce that $v_k\in L^\infty(B_1)$.

By $(V)$, for any $M>0$, there exists $R>0$ such that
\[V(x)\geq M>0\quad\text{for}~~|x|>R\text{.}\]
For each $m\in\mathbb{N}$ and $\gamma>1$, fix
\[A_m=\{x\in B_R^c: |v_k|^{\gamma-1}\leq m\}\text{,\quad} D_m=B_R^c\backslash A_m\]
and
\begin{align*}
v_m=
\left\{
\begin{array}{lll}
 v_k|v_k|^{2(\gamma-1)}\text{,} &\quad\text{in}~~A_m\text{,}\\
 m^2v_k\text{,}&\quad \text{in}~~D_m\text{,}\\
 0\text{,}&\quad \text{in}~~ B_R\text{.}
\end{array}
\right.
\end{align*}
Note that $v_m\in H^1(\R)$, $v_m\leq |v_k|^{2\gamma-1}$ and
\begin{align}\label{vm}
  \nabla v_m= \left\{
  \begin{array}{lll}
   (2\gamma-1)|v_k|^{2(\gamma-1)}\nabla v_k\text{,}&\quad\text{in} ~~A_m\text{,}\\
   m^2\nabla v_k\text{,}&\quad\text{in}~~D_m\text{,}\\
   0\text{,}&\quad\text{in}~~B_R\text{.}
  \end{array}
  \right.
\end{align}
For $\gamma>1$, we use $v_m$ as a test function in $\langle \mathcal{J}_k'(v_k), v_m\rangle=0$ to obtain
\begin{align}\label{text}
  \rn \left(\nabla v_k\cdot\nabla v_m+V(x)\frac{G^{-1}(v_k)}{g(G^{-1}(v_k))}v_m\right)=\rn \frac{f(G^{-1}(v_k))}{g(G^{-1}(v_k))}v_m\text{.}
\end{align}
By \eqref{vm}, we have
\begin{align}\label{vm1}
  \rn \nabla v_k\cdot\nabla v_m=(2\gamma-1)\int_{A_m}|v_k|^{2(\gamma-1)}|\nabla v_k|^2+m^2\int_{D_m}|\nabla v_k|^2\text{.}
\end{align}
Define
\begin{align*}
 w_m= \left\{
 \begin{array}{lll}
   v_k|v_k|^{\gamma-1}\text{,}&\quad\text{in}~~A_m\text{,}\\
   mv_k\text{,}&\quad\text{in}~~D_m\text{,}\\
   0\text{,}&\quad\text{in}~~B_R\text{.}
 \end{array}
  \right.
\end{align*}
Then, we have $w_m^2=v_{k}v_m\leq |v_{k}|^{2\gamma}$ and
\begin{align*}
  \nabla w_m=\left\{
  \begin{array}{lll}
    \gamma|v_{k}|^{\gamma-1}\nabla v_{k}\text{,}&\quad\text{in} ~~A_m\text{,}\\
    m\nabla v_{k}\text{,}&\quad\text{in}~~D_m\text{,}\\
    0,&\quad\text{in} ~~B_R\text{.}
  \end{array}
  \right.
\end{align*}
Subsequently,
\begin{align}\label{vm2}
  \rn |\nabla w_m|^2=\gamma^2\int_{A_m} |v_{k}|^{2(\gamma-1)}|\nabla v_{k}|^2+m^2\int_{D_m}|\nabla v_{k}|^2\text{.}
\end{align}
By \eqref{vm1} and \eqref{vm2}, we have
\begin{align}
  \label{text1}
  \rn \left(|\nabla w_m|^2-\nabla v_{k}\cdot\nabla v_m\right)=(\gamma-1)^2\int_{A_m}|v_{k}|^{2(\gamma-1)}|\nabla v_{k}|^2\text{.}
\end{align}
It follows from $(f_2)$ and $(f_3)$ that
\begin{align}
  \label{grow111}
  |f(t)|\leq \varepsilon|t|+C_\varepsilon |t|^{p-1}\text{,\quad} p\in(2,2^*)\text{.}
\end{align}
By \eqref{text1}, \eqref{vm1}, \eqref{text}, \eqref{grow111} and $(iii)$ of Lemma \ref{lem11}, we get
\begin{align*}
 \rn |\nabla w_m|^2&\leq\left[\frac{(\gamma-1)^2}{2\gamma-1}+1\right]\rn \nabla v_{k}\cdot\nabla v_m\\
 &\leq \gamma^2\rn \left(\nabla v_{k}\cdot\nabla v_m +(V(x)-\varepsilon)\frac{G^{-1}(v_{k})}{g(G^{-1}(v_{k}))}v_m\right)\\
 &=\gamma^2\rn \frac{f(G^{-1}(v_{k}))}{g(G^{-1}(v_{k}))}v_m-\gamma^2\rn \varepsilon \frac{G^{-1}(v_{k})}{g(G^{-1}(v_{k}))}v_m\\
 &\leq C \gamma^2\rn |v_{k}^{p-1}v_m|=C\gamma^2\int_{B_R^c} |v_{k}|^{p-2}w_m^2\text{.}
\end{align*}
By Sobolev imbedding and H\"{o}lder inequalities, we have
\begin{align*}
 \left( \int_{A_m}|w_m|^{2^*}\right)^{(N-2)/N}&\leq S^{-1}\int_{\mathbb{R}^N} |\nabla w_m|^2\leq CS^{-1}\gamma^2\int_{B_R^c} |v_{k}|^{p-2}w_m^2\\
 &\leq C\gamma^2S^{-1}|v_{k}|_{2^*}^{p-2}\left(\int_{B_R^c} |w_m|^{2p_1}\right)^{1/p_1}\text{,}
\end{align*}
where $1/p_1+(p-2)/2^*=1$. Since $|w_m|\leq |v_{k}|^\gamma$ in $B_R^c$ and $|w_m|=|v_{k}|^\gamma$ in $A_m$, we have
\begin{align*}
  \left(\int_{A_m}|v_{k}|^{\gamma 2^*}\right)^{(N-2)/N}\leq C\gamma^2S^{-1}|v_{k}|_{2^*}^{p-2}\left(\int_{B_R^c} |v_{k}|^{2\gamma p_1}\right)^{1/p_1}\text{.}
\end{align*}
Let $m\rightarrow\infty$, we see that
\begin{align}
  \label{gj}
  |v_{k}|_{L^{\gamma 2^*}(B_R^c)}\leq \gamma^{1/\gamma}\left(CS^{-1}|v_{k}|_{2^*}^{p-2}\right)^{1/(2\gamma)} |v_{k}|_{L^{2\gamma p_1}(B_R^c)}\text{.}
\end{align}
Now we can carry out an iteration process. Taking $\tau=2^*/(2p_1)$ and $\gamma=\tau$ in \eqref{gj}, we obtain
\begin{align*}
  |v_{k}|_{L^{\tau 2^*}(B_R^c)}\leq \tau^{1/\tau}\left(CS^{-1}|v_{k}|_{2^*}^{p-2}\right)^{1/(2\tau)} |v_{k}|_{L^{2^*} (B_R^c)}\text{.}
\end{align*}
Therefore, as $\gamma=\tau^2$ in \eqref{gj}, we obtain
\begin{align}\label{gj2}
  |v_{k}|_{L^{\tau^2 2^*}(B_R^c)}&\leq \tau^{2/\tau^2}\left(CS^{-1}|v_{k}|_{2^*}^{p-2}\right)^{1/(2\tau^2)} |v_{k}|_{L^{\tau 2^*} (B_R^c)}\notag\\
  &\leq\tau^{1/\tau+2/\tau^2}\left(CS^{-1}|v_{k}|_{2^*}^{p-2}\right)^{1/2\left(\frac{1}{\tau}+\frac{1}{\tau^2}\right)} |v_{k}|_{L^{2^*} (B_R^c)}\text{.}
\end{align}
Iterating \eqref{gj} and $\gamma=\tau^i(i=1,2,\cdots)$, it concludes that
\begin{align*}
 |v_{k}|_{L^{\tau^i 2^*}(B_R^c)}\leq \tau^{\sum_{i=1}^j\frac{j}{\tau^i}}\left(CS^{-1}|v_{k}|_{2^*}^{p-2}\right)^{\frac{1}{2}\sum_{i=1}^j\frac{1}{\tau^i}}|v_{k}|_{L^{2^*} (B_R^c)}\text{.}
\end{align*}
Thus, by $v_{k}\in H^1$ and taking the limit of $j\rightarrow+\infty$, we get
\begin{align*}
  |v_{k}|_{L^\infty(B_R^c)}\leq\tau^{\frac{1}{(\tau-1)^2}}\left(CS^{-1}|v_{k}|_{2^*}^{p-2}\right)^{\frac{1}{2(\tau-1)}}|v_{k}|_{L^{2^*} (B_R^c)}\leq C\text{.}
\end{align*}
Together with the fact $v_k\in L^\infty(B_1)$, there exists a constant $C_0>0$ such that
\[\|v_k\|_\infty\leq C_0\text{.}\]
Then we complete the proof.
\end{proof}

\begin{proof}
[Proof of Theorem \ref{th1}]
We have proved that $\mathcal{J}_k$ satisfies the Cerami condition. By Lemma \ref{l2}, $\mathcal{J}_k$ has a local linking at $0$ with respect to the decomposition $E=E^+\oplus E^-$. Together with Proposition \ref{pro3}, we obtain
\[C_\nu(\mathcal{J}_k,0)\neq 0\]
for $\nu=\dim E^-.$ From Lemma \ref{l5}, we have
\[C_\nu(\mathcal{J}_k, \infty)\neq C_\nu(\mathcal{J}_k, 0)\text{.}\]
This conclude by Proposition \ref{pro2} that $\mathcal{J}_k$ has a nontrivial critical point $v_k$, i.e., $u_k=G^{-1}(v_k)$ is a nontrivial solution of \eqref{eq3}. Using Lemma \ref{lem11}-$(iii)$ and Lemma \ref{lemm}, for $\hat{k}=\frac{3}{4\alpha c_0^2}$ we obtain
\[\|u_k\|_\infty=\|G^{-1}(v_k)\|_\infty\leq\frac{2}{\sqrt{3}}\|v_k\|_\infty\leq\sqrt{\frac{1}{\alpha k}}\text{,\quad}\forall k\in(0,\hat{k}]\text{.}\]
Thus, $u_k=G^{-1}(v_k)$ is a nontrivial solution of \eqref{eq1}.
\end{proof}

\section{Multiplicity result}

The proof of Theorem \ref{th21} relies on the following two key lemmas. Furthermore, we shall establish Theorem \ref{th2} by applying the following symmetric mountain pass theorem due to Ambrosetti-Rabinowitz \cite{MR0370183}.

\begin{lem}
  \label{lem111}
  For any $k\in(0,1)$ fixed, under the assumptions of Theorem \ref{th1}, the functional $\mathcal{J}_k$ admits a sequence of critical points $\{v_{j}^{k}\}_{j\geq1}$ with the following properties: there eixst $\beta_j>\alpha_j$ both of which are independent of $k>0$,  $\alpha_j\rightarrow\infty$ as $j\rightarrow\infty$ and $\mathcal{J}_k(v_j^k)\in[\alpha_j,\beta_j]$.
\end{lem}

\begin{pf}
Let $E_n=\text{span}\{\varphi_1\text{,}\cdots\text{,}\varphi_n\}$. Then, $E=E_n\oplus E_n^{\perp}$ for $n=1\text{,}2\text{,}\cdots$, where $E_n^{\perp}$ is orthogonal to $E_n$ in $E$. Define
\[\Upsilon_n:=\{H\in C(\Sigma_{R_n}\cap E_n, E): H\text{~~is odd and ~~} H=\text{id on~~} \partial\Sigma_{R_n}\cap E_n\}\]
where $R_n$ is chosen by based on what follows in the proof and
\[\Sigma_R=\{u\in E\text{,\quad} \|u\|\leq R\text{~~with~~} R>0\}\text{.}\]
By $(f_3)$, it is clear that there exists a positive constant $C$ such that
\begin{align}\label{F1}
 F(s)\geq C|s|^\mu\text{,\quad for~~}|s|\geq 1\text{.}
\end{align}
Combining this result with  $(iii)$ of Lemma \ref{lem11}, for $v\in E_n\setminus \Sigma_{R_n}$ we obtain
\begin{align*}
  \mathcal{J}_k(v)&\leq\frac{1}{2}\rn \left(|\nabla v|^2+\frac{4}{3}\widetilde{V}(x)v^2\right)-\rn \widetilde{F}(G^{-1}(v))\\
  &\leq\frac{1}{2}\rn \left(|\nabla v|^2+\frac{4}{3}\widetilde{V}(x)v^2\right)-C\rn  |G^{-1}(v)|^{\mu}\\
  &\leq\frac{1}{2}\rn \left(|\nabla v|^2+\frac{4}{3}\widetilde{V}(x)v^2\right)-C\rn  |v|^{\mu}<0\text{,}
\end{align*}
where $R_n>0$ is chosen such that the last inequality holds due to the equivalence of norms in finite-dimensional spaces.

Subsequently, we may choose $R_n$ independent of $k>0$ such that $R_n>\rho_n$( see the following \eqref{rho1} for the definition of $\rho_n$). Set
\[
\Gamma_j := \left\{
    H(\Sigma_{R_n} \cap E_n \setminus Y) :
    \begin{aligned}
        & H \in \Upsilon_n\text{,} \ n \geq j\text{,} \ Y = -Y \subset \Sigma_{R_n} \cap E_n \\
        & \text{is open, and } \gamma(Y) \leq n - j
    \end{aligned}
\right\}
\]
where $\gamma(\cdot)$ is the genus. Define
\[c_j(k)=\inf_{B\in \Gamma_j}\sup_{v\in B} \mathcal{J}_k(v)\text{,\quad}j=\nu+1,\nu+2,\cdots\text{,}\]
where $\nu$ is the unique index such that $0\in(\lambda_\nu, \lambda_{\nu+1})$.

Our immediate goal is to establish that $c_j(k)$ are critical values of $\mathcal{J}_k$ and find intervals $[\alpha_j,\beta_j]$ with $\alpha_j<\beta_j$ containing $c_j(k)$, where $\alpha_j\rightarrow\infty$ as $j\rightarrow\infty$. Towards this goal, we break down the proof into the following steps.

\textbf{Step 1.} Following the idea in \cite[Lemma 2.4]{Liu2013}, we have an intersection property: if $0<\rho<R_n$ for all $n\geq j$, then for $B\in\Gamma_j$, we have
\begin{align}\label{intersection}B\cap \partial \Sigma_\rho\cap E_{j-1}^\perp\neq\emptyset\text{.}\end{align}

\textbf{Step 2.} We now estimate the upper bound  for $c_j(k)$.

By $(f_3)$, \eqref{F1} and $(iii)$ of lemma \ref{lem11}, we have
\begin{align}\label{phi}
  \mathcal{J}_k(v)&\leq \frac{1}{2}\rn \left(|\nabla v|^2+\frac{4}{3}\widetilde{V}(x)v^2\right)-\rn \widetilde{F}(G^{-1}(v))\notag\\
  &\leq\frac{1}{2}\rn \left(|\nabla v|^2+\frac{4}{3}\widetilde{V}(x)v^2\right)-\int_{|v|\geq 1} F(G^{-1}(v))\notag\\
  &\leq \frac{1}{2}\rn \left(|\nabla v|^2+\frac{4}{3}\widetilde{V}(x)v^2\right)-\int_{|v|\geq 1} |v|^\mu:=\Phi(v)\text{.}
\end{align}
Here, we define
\[\beta_j = \inf_{B \in \Gamma_j} \sup_{v \in B} \Phi(v)\text{,~~} \quad j=\nu+1\text{,~~} \nu+2\text{,~~}\cdots\text{.}\]
It follows from \eqref{phi} that  $c_j(k)\leq \beta_j$.

\textbf{Step 3.} We shall prove the lower bound for $c_j(k)$.

\eqref{intersection} implies that \[c_j(k)\geq\inf_{v\in \partial\Sigma_\rho\cap E_{j-1}^\perp}\mathcal{J}_k(v)\text{.}\]
 For $v\in \partial \Sigma_{\rho}\cap E_{j-1}^\perp$ with $j>\nu$, by \eqref{5x}, \eqref{s1}, $|G^{-1}(t)|\geq |t|$ and the H\"{o}lder inequality we have
\begin{align*}
  \mathcal{J}_k(v)
  &\geq \kappa\rn\left( |\nabla v|^2+V(x)v^2\right)-\varepsilon\rn v^2-C_\varepsilon\rn |v|^p\\
  &\geq \kappa\rn\left( |\nabla v|^2+V(x)v^2\right)-\varepsilon\rn v^2-C_\varepsilon |v|_{2^*}^{ap}|v|_2^{(1-a)p}\\
  &\geq \left(\frac{\lambda_{\nu}\kappa}{\lambda_{\nu}+1}-\varepsilon\right)\rho^2 -C_\varepsilon S^{\frac{-ap}{2}}\lambda_{j}^{-\frac{(1-a)p}{2}}\rho^p\\
  &\geq \rho^2\left[\frac{\lambda_\nu\kappa}{2(\lambda_\nu+1)}-C\lambda_{j}^{-\frac{(1-a)p}{2}}\rho^{p-2}\right]\text{,}
\end{align*}
where $a\in(0,1)$ satisfies $\frac{1}{p}=\frac{a}{2^*}+\frac{1-a}{2}$ and $\varepsilon=\frac{\lambda_\nu\kappa}{2(\lambda_\nu+1)}$. Thus, we may choose $\rho=\rho_j$ such that
\begin{align}
  \label{rho1}
  C\lambda_{j}^{-\frac{(1-a)p}{2}}\rho_{j}^{p-2}=\frac{\lambda_\nu\kappa}{4(\lambda_\nu+1)}\text{.}
\end{align}
Subsequently, $\mathcal{J}_k(v)\geq\frac{\lambda_\nu\kappa}{4(\lambda_\nu+1)}\rho_j^2$, and we obtain \[\alpha_j=\frac{\lambda_\nu\kappa}{4(\lambda_\nu+1)}\rho_j^2=C\left(\lambda_j^{\frac{(1-a)p}{2}}\right)^{\frac{2}{p-2}}\rightarrow+\infty\] as $j\rightarrow\infty$.

\textbf{Step 4.} There exists a $j^*\in\mathbb{N}$ with $j>j^*$ such that $c_j(k)$, $j=j^*+1, j^*+2,\cdots$, are critical values of $\mathcal{J}_k$.

If $c_j(k)$ is not a critical value of $\mathcal{J}_k$. By the deformation theorem \cite{MR1400007} (which in particular holds under the Cerami condition \cite{Bartolo1983}),  for $0<\overline{\varepsilon}<\min\{\alpha_j: j=\nu+1,\nu+2,\cdots\}$, then there exists an $\varepsilon\in (0,\overline{\varepsilon})$ and $\eta\in C([0,1]\times E, E)$ such that
\begin{enumerate}[$(i)$]
  \item $\eta(0,v)=v$ for all $v\in E$.
  \item $\eta(t,v)=v$ for all $t\in[0,1]$ if $\mathcal{J}_k(v)\not\in [c_j(k)-\overline{\varepsilon},c_j(k)+\overline{\varepsilon} ]$.
  \item  $\eta(1, A_{c_j(k)+\varepsilon}\setminus O)\subset A_{c_j(k)-\varepsilon}$, where $A_c=\{v\in E: \mathcal{J}_k(v)\leq c\}$ and $O$ is any neighborhood of $K_c:=\{v\in E: \mathcal{J}_k(v)=c\text{\quad and\quad} \mathcal{J}_k'(v)=0\}$.
  \item $\eta(t,v)$ is odd in $v$.
\end{enumerate}

 Here, it is clear that $\varphi=\eta(1,\cdot): E\rightarrow E$ is odd, and $\varphi=id$ on $\partial \Sigma_{R_n}\cap E_n$ for all $n\geq j$ since $\mathcal{J}_k(v)<0$ for $v\in\partial \Sigma_{R_n}\cap E_n$. According to the definition of $c_j(k)$, there exists $B\in \Gamma_j$ such that
 \[\sup_{v\in B}\mathcal{J}_k(v)\geq c_j(k)+\varepsilon\text{.}\]
 It follows from $(ii)$ and $A=\varphi(B)\in \Gamma_j$ that
 \[c_j(k)\leq\sup_{v\in A}\mathcal{J}_k(v)\leq c_j(k)-\varepsilon\text{,}\]
 which is a contradiction. Based on the reasoning above, we have shown that $\mathcal{J}_k$ has a sequence of critical points $\{v_{k,j}\}$ with $\mathcal{J}_k(v)\in[\alpha_j, \beta_j]$, and $\alpha_j\rightarrow\infty$ as $j\rightarrow\infty$.

\end{pf}

\begin{lem}\label{lemn}
  For any $m\in\mathbb{N}$ fixed. If $v_{k,j}$ is a critical point of $\mathcal{J}_k$ where $j=j^*+1, j^*+2,\cdots,j^*+m$, then there exists $M>0$ independent of $k$, such that $\|v_{k,j}\|\leq M$.
\end{lem}

\begin{proof}
  If the conclusion is not true, there exist a sequence $\{k_n\}\subset(0,1)$ and $v_{k_n}$ is a critical point of $\mathcal{J}_{k_n}$ with $c_{j}(k_n)=\mathcal{J}_{k_n}(v_{k_n})\in[\alpha_j,\beta_j]$, such that, along with a subsequence, $\|v_{k_n}\|\rightarrow\infty$ as $n\rightarrow\infty$ and $\lim_{n\rightarrow\infty}k_n=k_0$. Similar to the argument of Lemma \ref{l1}. For the convenience of readers, we rewrite $v_{k_n}=v_n$. It is easy to see that
\begin{align*}
\mu\beta_j\geq\mu c_j(k_n)&=\mu\mathcal{J}_{k_n}(v_n)-\langle \mathcal{J}_{k_n}'(v_n), \psi_n\rangle\\
     &\geq\left(\frac{\mu}{2}-1\right)\varrho_n^2-\left(\frac{\mu}{2}-1\right)\rn m |G^{-1}(v_n)|^2
  \end{align*}
  where
  \[\varrho_n=\left(\rn \left(|\nabla v_n|^2+\widetilde{V}(x)|G^{-1}(v_n)|^2\right)\right)^{\frac{1}{2}}\rightarrow\infty\] as $n\rightarrow\infty$. Let $\theta_n=\frac{G^{-1}(v_n)}{\varrho_n}$. By \eqref{t2}, it is obvious that $\{\theta_n\}$ is bounded in $E$ and $|\theta_n|_2^2>0$. This implies that $\theta_n(x)\rightarrow\theta(x)$ a.e. in $\mathbb{R}^N$ and the set $\Lambda=\{x\in\mathbb{R}^N: \theta(x)\neq 0\}$ has a positive Lebesgue measure. By \eqref{f1}, we get
\begin{align*}
   \alpha_{j^*}
   \leq\mathcal{J}_{k_n}(v_n)&=\frac{1}{2}\rn\left( |\nabla v_n|^2+\widetilde{V}(x)|G^{-1}(v_n)|^2\right)-\rn \widetilde{F}(G^{-1}(v_n))\\
    &=\zeta_n^2\left(\frac{1}{2}-\rn \frac{\widetilde{F}(G^{-1}(v_n))}{\varrho_n^2}\right)\\
    &\leq\varrho_n^2\left(\frac{1}{2}-\int_{\Lambda} \frac{\widetilde{F}(G^{-1}(v_n))}{|G^{-1}(v_n)|^2}\theta_n^2\right)\rightarrow-\infty\text{,}
  \end{align*}
  which is a contradiction. Thus, we finish the proof.
\end{proof}

Next, we consider the limiting behavior of critical points of $\mathcal{J}_k$ as $k\rightarrow 0^+$.
\begin{prop}
  \label{th1-1}
  Assume that $k_n\rightarrow 0^+$, $\{v_n\}\subset E$ is a sequence of critical points of $\mathcal{J}_{k_n}$ satisfying $\mathcal{J}_{k_n}'(v_n)=0$ and $\mathcal{J}_{k_n}(v_n)\leq C$ for some $C$ independent of $n$. Then, up to a subsequence, there exist some  $v$ such that $v_n\rightarrow v$ in $E$ and $v$ is a  solution for the following elliptic equation
  \begin{align*}
    -\Delta u+V(x)u=f(u)\text{,}\quad u\in H^1(\mathbb{R}^N).
  \end{align*}
\end{prop}

\begin{proof}

Similar to the proof of Lemma \ref{l1}, by $\mathcal{J}_{k_n}'(v_n)=0$ and $\mathcal{J}_{k_n}(v_n)\leq C$ we prove that $\{v_n\}$ is bounded in $E$.  Using the changing variables $u=G^{-1}(v)\in E$ and \eqref{g1}, we have
\begin{align}\label{g2}
  C&\geq\int_{\mathbb{R}^N}\left(|\nabla v_n|^2+\widetilde{V}(x)|G^{-1}(v_n)|^2\right)\notag\\
  &=\int_{\mathbb{R}^N}g^2(u_n)|\nabla u_n|^2+\widetilde{V}(x)u_n^2\notag\\
  &=\int_{\mathbb{R}^N}(1-k_nu_n^2)|\nabla u_n|^2+\widetilde{V}(x)u_n^2\text{.}
\end{align}
This yields the boundedness of $\{u_n\}$ in $E$. This, together with the  compactness of the embedding $E\hookrightarrow L^p$ with $p\in[2,2^*)$, up to subsequence, there exists a $u_0\in E$ such that $u_n\rightarrow u_0$ in $L^p$, $p\in[2,2^*)$. Furthermore, the continuity of $E\hookrightarrow L^p$ for all $p\in[2,2^*]$, combined with Moser's  iteration in lemma \ref{lem41} and \eqref{g2}, implies that
\begin{align}\label{g3}
  \|v_n\|_\infty\leq C\text{,\quad}\|v\|_\infty\leq C
\end{align}
for some constant $C$ independent of $n$.

Now, we show that the functional $\mathcal{I}_0: E\rightarrow\mathbb{R}$,
\begin{align*}
  \mathcal{I}_0(u)=\frac{1}{2}\int_{\mathbb{R}^N}|\nabla u|^2+V(x)u^2-\int_{\mathbb{R}^N}F(u)
\end{align*}
admits a critical point $u_0\in E$. Firstly, we choose $\varphi=g(G^{-1}(v_n))\psi$, where $\psi\in C_0^\infty(\mathbb{R}^N)$. Testing $\mathcal{J}_{k_n}'(v_n)=0$ with $\varphi$, by $u_n=G^{-1}(v_n)$ and \eqref{g3} we have
\begin{align}\label{limitequation}
  0&=\langle \mathcal{J}_{k_n}'(v_n),\varphi\rangle\notag\\
  &=\int_{\mathbb{R}^N} g(u_n)g'(u_n)|\nabla u_n|^2\psi+\int_{\mathbb{R}^N} g^2(u_n)\nabla u_n\cdot\nabla\psi+\int_{\mathbb{R}^N} V(x)u_n\psi-\int_{\mathbb{R}^N} f(u_n)\psi\notag\\
  &=-k_n\int_{\mathbb{R}^N} u_n|\nabla u_n|^2\psi +\int_{\mathbb{R}^N}(1-k_n u_n^2)\nabla u_n\cdot \nabla\psi+\int_{\mathbb{R}^N} V(x)u_n\psi-\int_{\mathbb{R}^N} f(u_n)\psi\notag\\
  &=\int_{\mathbb{R}^N}\nabla u_n\cdot\nabla\psi+V(x)u_n\psi-\int_{\mathbb{R}^N} f(u_n)\psi+o(1)\notag\\
  &=\langle I_0'(u_0),\psi\rangle+o(1)\text{,}
\end{align}
where the following estimate holds
\begin{align*}
 -k_n\int_{\mathbb{R}^N} u_n|\nabla u_n|^2\psi -k_n\int_{\mathbb{R}^N} u_n^2\nabla u_n\cdot \nabla\psi\rightarrow 0\text{,\quad~~as~~}n\rightarrow\infty\text{.}
\end{align*}
In fact, by \eqref{g3} we also obtain that $\|u_n\|_{\infty}\leq C$ and $\|u_0\|_\infty\leq C$. Subsequently, \[\int_{\mathbb{R}^N}u_n|\nabla u_n|^2\psi\leq \|u_n\|_{\infty}\|\psi\|_{\infty}\int_{\mathbb{R}^N}|\nabla u_n|^2\leq C\text{.}\] It is clear that $k_n\int u_n|\nabla u_n|^2\psi\rightarrow 0$ and $k_n\int u_n^2\nabla u_n\cdot\nabla\psi\rightarrow 0$ similarly. That is, $u_0$ is a critical point of $\mathcal{I}_0$ and a solution of Eq. \eqref{eq33}. By performing additional approximation steps, we may have $u_0$ in the place of $\psi$ of \eqref{limitequation}:
\begin{align}\label{limiteq2}
  \int_{\mathbb{R}^N}|\nabla u_0|^2+V(x)u_0^2-\int_{\mathbb{R}^N}f(u_0)u_0=0\text{.}
  \end{align}
Setting $\varphi=g(G^{-1}(v_n))G^{-1}(v_n)$, using  $u_n=G^{-1}(v_n)$ and $\langle \mathcal{J}_{k_n}'(v_n),\varphi\rangle=0$ we obtain
\begin{align}\label{yuaneq}
 0&= \langle \mathcal{J}_{k_n}'(v_n),\varphi\rangle\notag\\
 &=-2k_n\int_{\mathbb{R}^N} u_n^2|\nabla u_n|^2 +\int_{\mathbb{R}^N}|\nabla u_n|^2+\int_{\mathbb{R}^N} V(x)u_n^2-\int_{\mathbb{R}^N} f(u_n)u_n\text{.}
\end{align}
Using the fact of $u_n\rightarrow u_0$ in $L^p$ with $p\in[2,2^*)$, it is clear that
\begin{align*}
  \int_{\mathbb{R}^N}f(u_n)u_n\rightarrow \int_{\mathbb{R}^N}f(u_0)u_0\text{,\quad} \int_{\mathbb{R}^N}V(x)u_n^2\rightarrow\int_{\mathbb{R}^N}V(x)u_0^2\text{.}
\end{align*}
From this, \eqref{limiteq2}, \eqref{yuaneq}, $\|u_n\|_{\infty}\leq C$ and lower semi-continuity, we derive
\begin{align*}
  \int_{\mathbb{R}^N}|\nabla u_n|^2\rightarrow \int_{\mathbb{R}^N}|\nabla u_0|^2\text{,\quad}k_n\int u_n^2|\nabla u_n|^2\rightarrow 0\text{.}
\end{align*}
In particular, we have $u_n\rightarrow u_0$ in $E$ and $\mathcal{J}_{k_n}(v_n)\rightarrow I_0(u_0)$ with $u_n=G^{-1}(v_n)$.
\end{proof}

\begin{proof}
 [Proof of Theorem \ref{th21}]
 We prove Theorem \ref{th21} in two steps:

\textbf{Step 1}: For $k\in(0,1)$, Lemma \ref{lem111} implies that $\mathcal{J}_k$ has a sequence of critical points $\{v_{k,j}\}$,
\[\mathcal{J}_k(v_{k,j})=c_j(k)\geq\alpha_j\rightarrow\infty\text{,\quad as~~}j\rightarrow\infty\text{.}\]
Subsequently, according to Lemmas \ref{lem41} and \ref{lemn}, for any fixed $m\in\mathbb{N}$, there exists an $M>0$ such that
\begin{align*}
  \|v_{k,j}\|_{\infty}\leq M\text{,\quad where~~} j\in\{j^*+1,\cdots,j^*+m\}\text{.}
\end{align*}
Therefore, there exists $k_0>0$ small enough such that for any $0<k<k_0$, we have
\[\|u_{k,j}\|_\infty=\|G^{-1}(v_{k,j})\|_\infty\leq\frac{2}{\sqrt{3}}\|v_{k,j}\|_\infty
\leq\frac{2}{\sqrt{3}}M\leq\sqrt{\frac{1}{\alpha k}}\text{.}\]
Then, for any $k\in(0,k_0)$, $u_{k,j}=G^{-1}(v_{k,j})(j=j^*+1,\cdots,j^*+m)$ are solutions of the original Eq. \eqref{eq1}. Therefore, $(i)$ of Theorem \ref{th21} holds.

\textbf{Step 2}: We now prove part $(ii)$ of Theorem \ref{th21}.
Using proposition \ref{th1-1} and Lemma \ref{lem41}, as $k_n\rightarrow 0^+$, there exists a $u_{0,j}\in E$ we obtain
\[\alpha_j\leq c_j(k_n)=\mathcal{J}_{k_n}(v_{k_n,j})\rightarrow I_0(u_{0,j})\text{,\quad} \mathcal{I}_0'(u_{0,j})=0\]
where $u_{k_n,j}=G^{-1}(v_{k_n,j})$ and $u_{k_n,j}\rightarrow u_{0,j}$ in $E$. Subsequently, we obtain, as $k\rightarrow 0^+$, $u_{k,j}$ converges to a solution $u_{0,j}$ of problem \eqref{eq33} with $\mathcal{I}_0(u_{0,j})\rightarrow\infty$ as $j\rightarrow\infty$.
\end{proof}

To remove the condition $f$ on near zero, we need the following lemma to establish the existence of infinitely many nontrivial solutions for Eq.\eqref{eq1}.
\begin{prop}(\cite[Theorem 9.12]{MR765240})\label{pro4}
Let $E$ be an infinite dimensional Banach space, $\Phi\in C^1(E,\mathbb{R})$ be even, satisfies Cerami condition and $\Phi(0)=0$. If $E=Y\oplus Z$ with $\dim Y<\infty$, and $\Phi$ satisfies
\begin{enumerate}[$(a)$]
  \item there are constants $\rho, \alpha>0$ such that $\Phi|_{\partial B_\rho\cap Z}\geq \alpha$,
  \item for any finite dimensional subspace $W\subset E$, there is an $R=R(W)$ such that $\Phi\leq 0$ on $W\setminus B_R(W)$,
\end{enumerate}
then $\Phi$ has a sequence of critical values $c_j\rightarrow\infty$.
\end{prop}

\begin{lem}
  \label{lem51}
  Let $W$ be a finite dimensional subspace of $E$, then $\mathcal{J}_k$ is anti-coercive on $W$, that is
  \[\mathcal{J}_k(v)\rightarrow-\infty\text{,\quad} \text{as}~~\|v\|\rightarrow\infty\text{.}\]
\end{lem}

\begin{proof}
  For any $\{v_n\}\subset W$ with $\|v_n\|\rightarrow\infty$, we set $w_n=\frac{v_n}{\|v_n\|}$. Then $w_n$ is bounded in $W$. Because $\dim W<+\infty$, there exists $w\in W\backslash \{0\}$ such that
  \[\|w_n- w\|\rightarrow 0\text{,}\qquad w_n(x)\rightarrow w(x) ~~\text{a.e.}~~x\in\R\text{.}\]
  For $x\in \{w\neq 0\}$, we have $|v_n(x)|\rightarrow\infty$ and $|G^{-1}(v_n)|\rightarrow\infty$. Therefore, for $n$ large enough we have $|v_n(x)|\geq 1$. By $(f_3)$, we derive that
  \[\lim_{t\rightarrow\infty}\frac{F(t)}{t^2}=+\infty\text{.}\]
This togethers with $(iii)$ in Lemma \ref{lem11}, implies
  \begin{align}
    \label{infty}
    \frac{\widetilde{F}( G^{-1}(v_n))}{\|v_n\|^2}&=\frac{F(G^{-1}(v_n))}{ |G^{-1}(v_n)|^2}\frac{|G^{-1}(v_n)|^2}{v_n^2}w_n^2+\frac{m|G^{-1}(v_n)|^2}{2v_n^2}w_n^2\rightarrow\infty\text{.}
  \end{align}
  By  Lemma \ref{lem11}-$(iii)$, the Fatou lemma and \eqref{infty}, we obtain
  \begin{align*}
    \mathcal{J}_k(v_n)&=\frac{1}{2}\rn |\nabla v_n|^2+\widetilde{V}(x)(G^{-1}(v_n))^2-\rn \widetilde{F}(G^{-1}(v_n))\\
    &\leq \|v_n\|^2\left(\frac{2}{3}-\int_{w\neq 0}\frac{\widetilde{F}(G^{-1}(v_n))}{\|v_n\|^2}\right)\rightarrow -\infty\text{.}
  \end{align*}
  Then we finished the proof.
\end{proof}

\begin{proof}
[Proof of Theorem \ref{th2}]
Under the assumptions of Theorem \ref{th2}, we know that $\mathcal{J}_k$ is an even functional satisfying the Cerami condition. Lemma \ref{lem51} implies that $\mathcal{J}_k$ satisfies condition $(b)$ of Proposition \ref{pro4}.
 For any $\nu$, let $Z_\nu=\overline{\text{span}}\{\varphi_\nu,\varphi_{\nu+1},\cdots\}$. Thanks to \cite[Lemma 2.5]{MR2548724}, for any $2\leq q<2^*$, we have that
 \begin{align}
   \label{52}
   \beta_\nu=\sup_{v\in Z_\nu,~~\|v\|=1}|v|_q\rightarrow 0\text{,}\qquad \text{as}~~\nu\rightarrow\infty\text{.}
 \end{align}
 Let
 \[
 Y=\text{span}\{\varphi_1,\cdots,\varphi_{\nu-1}\}\text{.}
 \]
 Then $Z_\nu\subset E^+$ and $E=Y\oplus Z_\nu$. For $v\in Z_\nu$, using Taylor expansion as in the proof of Lemma \ref{l2} and \eqref{grow}, we obtain
  \begin{align}\label{53}
    \mathcal{J}_k(v)&=\frac{1}{2}\rn |\nabla v|^2+V(x)G^{-1}(v)^2-\rn F(G^{-1}(v))\notag\\
    &=\frac{1}{2}\rn |\nabla v|^2+V(x)v^2-\rn F(G^{-1}(v))+o(\|v\|^2)\notag\\
    &\geq \kappa\|v\|^2-\varepsilon |v_n|_2^2-C_\varepsilon|v|_p^p+o(\|v\|^2)\notag\\
    &\geq \kappa\|v\|^2-\varepsilon \beta_\nu^2\|v\|^2-C_\varepsilon\beta_\nu^p\|v\|^p+o(\|v\|^2)\notag\\
    &\geq\vartheta\|v\|^2+o(\|v\|^2)\text{,}
  \end{align}
  where we take $\varepsilon$ small enough such that $\kappa-\varepsilon \beta_\nu^2\geq \vartheta>0$.  If $v\in Z_\nu\cap \partial B_\rho$, by \eqref{53}, there exists a constant $\alpha>0$ such that $\mathcal{J}_k(v)\geq \alpha$. So, $(a)$ of Proposition \ref{pro4} is verified. Therefore, $\mathcal{J}_k$ has a sequence of critical points $\{v_n\}$ such that $\mathcal{J}_k(v_n)\rightarrow\infty$. Let $u_n=G^{-1}(v_n)$, Lemma \ref{lem41} implies that $\{u_n\}$ is a sequence of solutions for Eq. \eqref{eq1} such that $\mathcal{I}(u_n)\rightarrow\infty$ because $\mathcal{J}_k(v_n)=\mathcal{I}_k(G^{-1}(v_n))$
\end{proof}

\section*{Acknowledgements}
L. Yin was supported by the Natural Science Foundation of Sichuan Province (No. 2024NSFSC1343) and National Natural Science Foundation of China (No. 12401140). X. Liu was supported by the National Natural Science Foundation of China(No. 12401128), and Y. Li was supported by the National Natural Science Foundation of China(No. 12301143).


\providecommand{\href}[2]{#2}
\providecommand{\arxiv}[1]{\href{http://arxiv.org/abs/#1}{arXiv:#1}}
\providecommand{\url}[1]{\texttt{#1}}
\providecommand{\urlprefix}{URL }

\end{document}